\documentclass[11pt]{amsart} 
\usepackage{euler}
\usepackage{euscript}
\usepackage{multicol} 
\usepackage{amssymb}
\usepackage{amsmath} 
\usepackage{times} 

\newcommand{\tr}{\mbox{\upshape tr\ }}
\newtheorem{theorem}{Theorem}[section] 
\newtheorem{lemma}[theorem]{Lemma}
\newtheorem{proposition}[theorem]{Proposition}
\newtheorem{corollary}[theorem]{Corollary}

\newtheorem{assumption}[theorem]{Assumption}
\numberwithin{equation}{section}
  \newcommand{\rz}{{\mathbb{R}}} 
  \newcommand{\cz}{{\mathbb{C}}} 
  \newcommand{\nz}{{\mathbb{N}}} 
  \newcommand{\rot}{{\rm curl}}

\def\tr{\mathop{\mathrm{tr}} \nolimits} 

\newcommand{\eps}{\varepsilon}

\theoremstyle{definition}
\newtheorem{remark}[theorem]{Remark}

\IfFileExists{amsthm.sty}
{ 
  \usepackage{amsthm}
  \def\qed{\hbox {\hskip 1pt \vrule width 6pt height 6pt depth 1.5pt
        \hskip 1pt}}
}
{ 
\def\qed{\hbox {\hskip 1pt \vrule width 6pt height 6pt depth 1.5pt
        \hskip 1pt}}
\newenvironment{proof}{\textit{Proof.}}{\qed \flushright}
}

\title[Improved Berezin-Li-Yau inequalities with magnetic field]
{Improved Berezin-Li-Yau inequalities with magnetic field}

\author{Hynek Kova\v r\'{\i}k}
\address{Hynek Kova\v r\'{\i}k, Dipartimento di Matematica, Universit\`a degli studi di Brescia,
Via Branze, 38 - 25123  Brescia, Italy}
\email{hynek.kovarik@ing.unibs.it}

\author{Timo Weidl}
\address{Timo Weidl, Institute of Analysis, Dynamics
and Modeling, Universit\"at Stuttgart, 
PF 80 11 40, D-70569  Stuttgart, Germany}
\email{weidl@mathematik.uni-stuttgart.de}

\begin{document}

\begin{abstract}
In this paper we study the eigenvalue sums of Dirichlet Laplacians on bounded domains. Among our results  we establish an improvement of the Li-Yau bound in the presence of a constant magnetic field previously obtained in \cite{ELV}. 
\end{abstract}

\maketitle
 
\maketitle


\section{\bf Introduction}
\label{sec:intro}
Let $\Omega\subset\rz^d$ be an open bounded domain.  We consider the Dirichlet
Laplacian $-\Delta_\Omega$ on $L^2(\Omega)$ defined in the quadratic form sense. Since the
embedding $H^1_0\hookrightarrow L^2(\Omega)$ is compact, 
the spectrum of the non-negative operator $-\Delta_\Omega$
is discrete and accumulates to infinity only.
Denote by $\{\lambda_j\}_{j\in\nz}=\{\lambda_j(\Omega)\}_{j\in\nz}$ the inreasing sequence of the eigenvalues of 
$-\Delta_\Omega$, where we repeat entrees according to their multiplicity. 

In particular, we shall study the so-called Riesz means of these eigenvalues, given by
.\footnote{For $\gamma=0$ this is simply the counting function 
of all eigenvalues $\lambda_j(\Omega)<\Lambda$.}
$$
\tr (-\Delta_\Omega-\Lambda)_-^\gamma =\sum_k \, ( \Lambda-\lambda_k)_+^ \gamma\ , \qquad \gamma\geq 0.
$$
Here and below we use the notation $x_\pm=(|x|\pm x)/2$. It is well-known that these Riesz means satisfy
the Weyl asymptotics  \cite{We}
\begin{equation} \label{weyl}
\sum_k \, ( \Lambda-\lambda_k)_+^ \gamma = L_{\gamma,d}^{cl}\, |\Omega|\, \Lambda^{\gamma+\frac d2} + o(\Lambda^{\gamma+\frac d2} ), \qquad \Lambda \to\infty,
\end{equation}
where 
$$
L_{\gamma,d}^{cl} = \frac{\Gamma(\gamma+1)}{(4\pi)^{\frac d2}\, \Gamma(\gamma+1+d/2)}.
$$
In 1972 Berezin \cite{Be} showed that for $\gamma \geq 1$ the leading term in \eqref{weyl} gives actually a uniform upper bound on the Riesz means,
namely for any  $\gamma\geq 1$ it holds
\begin{equation} \label{berezin}
\sum_k \, ( \Lambda-\lambda_k)_+^ \gamma \, \leq \,  L_{\gamma,d}^{cl}\, |\Omega|\, \Lambda^{\gamma+\frac d2}\, .
\end{equation}
In view of the asymptotics \eqref{weyl} the constant on the r.h.s. of \eqref{berezin}  is optimal. 
The bound \eqref{berezin} is assumed to hold for all $0\leq\gamma<1$ as well. However,
so far
this has been shown for tiling domains \cite{Po} and cartesian products with tiling domains
\cite{Lap} only. On the other hand, it follows from \eqref{berezin} that a similar inequality
holds for arbitrary domains and for all $0\leq\gamma<1$ with some probably non-sharp
excess factor on the r.h.s. \cite{Lap}
\begin{equation} \label{berezin1}
\sum_k \, ( \Lambda-\lambda_k)_+^ \gamma \, \leq \,  2\left(\frac{\gamma}{\gamma+1}\right)^\gamma
L_{\gamma,d}^{cl}\, |\Omega|\, \Lambda^{\gamma+\frac d2}, \qquad 0\leq \gamma<1.
\end{equation}

\smallskip
Here we are going to focus on the border-line case $\gamma=1$, in which the inequality \eqref{berezin} is equivalent, via Legendre transformation, to the lower bound
\begin{equation} \label{li-yau-class}
\sum_{j=1}^N \lambda_j(\Omega) \ \geq \ C_d\ |\Omega|^{-\frac 2d}\ N^{1+\frac 2d}, \qquad C_d = \frac{4\, \pi\,  d}{d+2}\ \Gamma(d/2+1)^{\frac 2d} .
 \end{equation}
The above estimate was proved in \cite{LY} independently on \cite{Be} and it is known as the Li-Yau inequality. Similarly as in the case of Berezin inequality, the constant $C_d$ cannot be improved, since the right hand side of \eqref{li-yau-class} gives the leading term of the Weyl asymtotic formula, see \eqref{weyl-2term} below. 

\smallskip
However,  the bounds \eqref{berezin} and \eqref{li-yau-class} can 
be improved by adding to its right hand side reminder term of a lower order in $\Lambda$ or in $N$, 
respectively. 
Several results in this direction were obtained recently both for the Berezin inequality \cite{GLW,W} (for $\gamma\geq\frac32$) and for the Li-Yau estimate \cite{FLU, Me, KVW,Y,YY}. 
In particular, Melas proved in \cite {Me} that there exists a positive constant $M_d$ such that
\begin{equation} \label{melas}
\sum_{j=1}^N \lambda_j(\Omega) \ \geq \ C_d\ |\Omega|^{-\frac 2d}\ N^{1+\frac 2d} + M_d \, \frac{|\Omega|}{I(\Omega)}\  N, \quad I(\Omega) = \min_{a\in\rz^d} \int_\Omega |x-a|^2\, dx,
\end{equation}
where $M_d\geq \frac{1}{24(d+2)}$.

Alongside with the ordinary Dirichlet Laplacian we shall also consider its magnetic version
$H(A)=(i\nabla+A(x))^2$ on $L^2(\Omega)$ generated by the closed quadratic form 
\begin{equation} \label{q-form} 
\| (i\nabla+A) \, u \|_{L^2(\Omega)}^2\, , \qquad u\in H^1_0(\Omega)\,,
\end{equation} 
where $A$ is a real-valued vector potential satisfying mild regularity conditions. 
Moreover, the magnetic 
Sobolev norm on the bounded domain $\Omega$ is equivalent to the non-magnetic one and the
operator $H(A)$ has discrete spectrum as well.  We notate its eigenvalues by 
$\lambda_k=\lambda_k(\Omega;A)$, repeating eigenvalues according to their multiplicities.
Note that the magnetic Riesz means satisfy the very same Weyl asymptotics \eqref{weyl}.

From the pointwise diamagnetic inequality (see e.g. \cite[Thm.7.21]{LL})
\begin{equation} \label{diamag}
\big | \nabla\, |u(x)|\, \big | \ \leq \  \big | (i\nabla+A)\, u(x)\, \big | \qquad \text{a.  e. }\quad x\in\Omega,
\end{equation}
it  follows that $\lambda_1(\Omega;A) \geq \lambda_1(\Omega;0)=\lambda_1(\Omega)$. 
However, the estimate $\lambda_j(\Omega;A) \geq \lambda_j(\Omega;0)=\lambda_j(\Omega)$
fails in general if $ j\geq 2$. Therefore, it is a priori not  clear whether bounds similar to \eqref{berezin}-\eqref{melas}  
remain true when the eigenvalues $\lambda_j(\Omega)$ are replaced by their magnetic counterparts $\lambda_j(\Omega;A)$.

\smallskip

By now it has been shown that
\begin{itemize}
\item the sharp bound \eqref{berezin} holds true for arbitrary magnetic fields if $\gamma\geq\frac32$ (\cite{LW}),
\item the sharp bound \eqref{berezin} holds true for constant magnetic fields if $\gamma\geq 1$ (\cite{ELV}),
\item in the dimension $d=2$ the bound \eqref{berezin1} holds true for constant magnetic fields if $0\leq \gamma<1$ and the constant
on the r.h.s. of \eqref{berezin1}  cannot be improved (\cite{FLW}) even in the class of constant magnetic fields
and tiling domains $\Omega$.
\end{itemize}
So far it is not known, whether the bound  \eqref{berezin} holds true for arbitrary magnetic fields if  
$1\leq \gamma\leq \frac32$. 

For $\gamma=1$ and constant magnetic field the magnetic version of \eqref{berezin}
is again dual to the magnetic version of the Li-Yau bound \eqref{li-yau-class}. Since \eqref{berezin} fails without
excess factor for all $\gamma<1$, the case $\gamma=1$ is the threshold case, in which the Berezin bound
with the classical constant remains true. 
Therefore it is of a particular interest to study, whether either the magnetic Berezin bound for $\gamma=1$
or equivalently the magnetic Li-Yau bound admits any further improvement by lower order remainder terms.

\smallskip
The main purpose of this paper is to establish an improved Li-Yau bound 
with an additional term of the Melas order for magnetic Dirichlet Laplacians on planar domains $\Omega \subset \rz^2$ with constant magnetic field.
For this end we first prove a different version of the Melas result in the non-magnetic case. 
Our proof is based on a new approach and yields the reminder term of the same order in $N$, i.e. linear, but 
with a different geometrical factor, see Theorem \ref{thm:li-yau} and Corollary \ref{li-yau-conv}. 
More importantly, in contrast to the classical Melas proof  our method extends to a lower bound for
the magnetic eigenvalues $\lambda_k(\Omega;A)$ as well, see Theorem \ref{thm:li-yau-mg} and Corollary \ref{cor:convex}.

\section{\bf Main results}
\label{sec:prelim}

\noindent 

\subsection{Preliminaries} Given a set $ \Omega \subset \rz^d$ we denote its volume by $|\Omega|$. Moreover, we denote by 
\begin{equation}\label{min}
\delta(x) = {\rm dist}\, (x,\partial\Omega) = \min_{y\in\partial\Omega} |x-y|
\end{equation}
the distance between a given $x\in\Omega$ and the boundary of $\Omega$, and by
$$
R_i(\Omega) = \sup_{x\in\Omega}\, \delta(x)
$$
the in-radius of $\Omega$. Given $\beta>0$ we introduce 
$$
\Omega_\beta = \{ x\in\Omega\, :\, \delta(x) < \beta\}, \qquad \beta >0,
$$
and define the quantity 
\begin{equation} \label{em}
\sigma(\Omega) := \inf_{0 <\beta < R_i(\Omega)}\, \frac{|\Omega_\beta|}{\beta}.
\end{equation} 
Note that $\sigma(\Omega) >0$ since the right hand side of \eqref{em} is a positive continuous function of $\beta$ and 
$$
\liminf_{\beta\to 0} \frac{|\Omega_\beta|}{\beta} >0.
$$
The quantity $\sigma(\Omega)$, which depends only on the geometry of $\Omega$, will play an important role in the sequel.  Throughout the paper we will suppose that $\Omega$ satisfies the following condition:

\begin{assumption} \label{ass-basic}
The domain $\Omega\subset\rz^d$ is open bounded and such that 
\begin{equation} \label{var-prob} 
\inf_{u\in H^1_0(\Omega)} \frac{ \int_\Omega |\nabla u|^2}{\int_\Omega\, |u|^2 /\delta^2} \,  =: \, c^{-1}_h(\Omega) > 0.
\end{equation} 
\end{assumption}

\noindent Note that  $c_h(\Omega)$ is the best constant in the Hardy inequality 
\begin{equation} \label{hardy}
\int_\Omega\, \frac{|u(x)|^2}{\delta(x)^2}\, dx \ \leq \ c_h(\Omega)\, \int_\Omega |\nabla u(x)|^2\, dx \qquad \forall\ u\in H^1_0(\Omega).
\end{equation} 

\medskip

\begin{remark} 
Assumption \ref{ass-basic} is satisfied, for example, for all open bounded domains with Lipschitz boundary, see \cite{A}. It is know that for simply connected planar domains $c_h(\Omega)\leq 16$, \cite{A}, and  for convex domains $c_h(\Omega)= 4$, see e.g. \cite{Da,MS},

\end{remark}

\subsection{Main results: Dirichlet Laplacian}
 
\begin{theorem} \label{thm:li-yau}
For any $N\in\nz$ we have 
\begin{equation} \label{li-yau-imp}
\sum_{j=1}^N \lambda_j(\Omega) \ \geq \ C_d\ |\Omega|^{-\frac 2d}\ N^{1+\frac 2d} + \frac{1}{16\, c_h(\Omega)}\ \frac{\sigma^2(\Omega)}{|\Omega|^2}\ N\ .
\end{equation}
\end{theorem}

\bigskip

\noindent For convex domains, in particular, we have 

\medskip

\begin{corollary} \label{li-yau-conv}
Let $\Omega\subset\rz^d$ satisfy assumption \ref{ass-basic} and suppose moreover that $\Omega$ is convex. Then for any $N\in\nz$ it holds
\begin{equation} \label{li-yau-convex}
\sum_{j=1}^N \lambda_j(\Omega) \ \geq \ C_d\ |\Omega|^{-\frac 2d}\ N^{1+\frac 2d} + \frac{N}{64\ R_i^2(\Omega)}\  .
\end{equation}
\end{corollary}

\smallskip

 \begin{remark} \label{rem1} 

Let us compare the lower bound \eqref{li-yau-convex} with \eqref{melas}. Assume that $a\in\rz^d$ is such that $I(\Omega)= \int_\Omega |x-a|^2\, dx$ and let $B(a,R)$ be the ball centered in $a$ with radius $R$ chosen such that $|B(a,R)|=|\Omega|$.  Then it is easily seen that 
\begin{align} \label{isoper}
I(\Omega) & \geq \ I(B(a,R)) = \frac{d}{d+2}\, |\Omega|\, R^2.
\end{align}
By using the fact that $R\geq R_i(\Omega)$ we thus obtain 
$$
\frac{1}{R_i^{2}(\Omega)} \geq   \frac{d}{d+2}\ \frac{|\Omega|}{I(\Omega)}.
$$
Hence, for convex $\Omega$, inequality \eqref{li-yau-convex} implies \eqref{melas}  with $M_d = \frac{d}{64 (d+2)}$. For $d\geq 3$ this is better than the lower bound $M_d \geq \frac{1}{24(d+2)}$ obtained in \cite{Me}.

\smallskip

On the other hand, for domains which are wide in one direction and thin in another the estimate \eqref{li-yau-convex} is much sharper than \eqref{melas} due to the fact that $\lambda_1(\Omega)$ is proportional to $R_i(\Omega)^{-2}$. Indeed, consider for example the rectangle $\Omega_\eps=(0, \eps^{-1})\times(0,\eps)$ in $\rz^2$. Then as $\eps\to 0$ we find $|\Omega_\eps|/I(\Omega_\eps) \, \sim \, 3\eps^2$, while on the right hand side of  \eqref{li-yau-convex} we have $R_i^{-2}(\Omega_\eps) = \eps^{-2}$ which is of the same  order of $\eps$ as the left hand side. 
\end{remark}

\begin{remark} \label{rem2}
The reminder terms in both bounds \eqref{li-yau-convex} and \eqref{melas} are not sharp in the order of $N$. This follows from the refined Weyl asymptotic 
\begin{equation} \label{weyl-2term}
\sum_{j=1}^N \lambda_j(\Omega) = C_d\ |\Omega|^{-\frac 2d}\ N^{1+\frac 2d} + K_d\ \frac{|\partial\Omega|}{|\Omega|^{1+\frac 1d}}\ N^{1+\frac 1d}(1 + o(1)) \qquad N\to \infty,
\end{equation}
with a positive constant $K_d$ depending only $d$. The asymptotic equation \eqref{weyl-2term} was first proven by Ivrii \cite{Iv,Iv2} for smooth domains under an additional assumption on the set of all periodic geodesic billiards in $\Omega$, see also \cite{SV}. Recently, \eqref{weyl-2term} was extended to all domains with $C^{1,\alpha}$ boundary (with $\alpha>0$) by Frank and Geisinger \cite{FG}.

\end{remark}

\smallskip

\subsection{Main results: magnetic Dirichlet Laplacian} 
As already mentioned in the introduction, our approach enables us to extend the bound \eqref{li-yau-imp} to the magnetic Dirichlet Laplacian. Let $B\in\rz$ be a non-zero constant define vector potential $A(x)=\frac 12 (- B x_2, B x_1)$ so that $\rot\, A = B$. 
We then have

\begin{theorem} \label{thm:li-yau-mg}
Let $d=2$. Then For any $N\in\nz$ it holds 
\begin{equation} \label{eq-li-yau-mg}
\sum_{j=1}^N \lambda_j(\Omega;A) \geq  \frac{2\pi \, N^2}{|\Omega|}+ \frac{1}{16\, c_h(\Omega)} \ \frac{\sigma^2(\Omega)}{|\Omega|^2}\ N \, .
\end{equation}
\end{theorem}

\medskip

\begin{corollary} \label{cor:convex}
Let $\Omega\subset\rz^2$ be bounded and convex. Then
\begin{equation} \label{eq-convex}
\sum_{j=1}^N \lambda_j(\Omega;A) \ \geq \  \frac{2\pi \, N^2}{|\Omega|} + \frac{N}{64\ R_i^2(\Omega)}\  .
\end{equation}
\end{corollary}

\medskip

\section{\bf Proofs of the main results}
\label{sec:non-mag}

\subsection{Dirichlet Laplacian} Given $\Lambda >0$ we denote by
$$
n(\Lambda) = \text{card} \big\{ \lambda_j(\Omega)\ : \ \lambda_j(\Omega) < \Lambda\big \}
$$
the counting function. Let  $\{u_j\}_{j\in\nz}$ be the set of eigenfunctions of $-\Delta_\Omega$ corresponding to the eigenvalues $\lambda_j(\Omega)$. We assume that the eigenfunctions are normalised in $L^2(\Omega)$ and denote by $\hat u_j(\xi)$ the Fourier transform of $u_j$ extended by zero to $\rz^d$;
\begin{equation}
\hat u_j(\xi) = (2\pi)^{-d/2}\, \int_\Omega e^{-ix\cdot\xi}\ u_j(x)\, dx.
\end{equation}
Then
\begin{align}
\sum_{j: \lambda_j(\Omega)< \Lambda}\ (\Lambda-\lambda_j(\Omega))
&=
\sum_{j \leq n(\Lambda)} \int_\Omega 
(\Lambda\, |u_j(x)|^2 -|\nabla u_j(x)|^2)\, dx = \sum_{j \leq n(\Lambda)}  \int_{\rz^d} (\Lambda-|\xi|^2)\, |\hat u_j(\xi)|^2\, d\xi 
\nonumber\\
&=\sum_{j\in\nz} \int_{\rz^d} (\Lambda-|\xi|^2)_+\, |\hat u_j(\xi)|^2\, d\xi \nonumber \\
&  \qquad -  \int_{\rz^d} (|\xi|^2- \Lambda)_+ \, R_1(\Lambda,\xi)\, d\xi  \, - \int_{\rz^d} (\Lambda- |\xi|^2)_+ \, R_2(\Lambda,\xi)\, d\xi , \label{1-eq}
\end{align}
where
$$
R_1(\Lambda,\xi) = \sum_{j \leq n(\Lambda)} |\hat u_j(\xi)|^2, \qquad R_2(\Lambda,\xi) = \sum_{j > n(\Lambda)} |\hat u_j(\xi)|^2.
$$
Since $\{u_j\}_{j\in\nz}$ is an orthonormal basis of $L^2(\Omega)$ and $\|e^{-i x\cdot\xi}\|^2_{L^2(\Omega)}=|\Omega|$, the Parseval identity implies 
\begin{equation} \label{parseval}
R_1(\Lambda,\xi)+R_2(\Lambda,\xi) = \sum_{j \in\nz} |\hat u_j(\xi)|^2 = (2\pi)^{-d}\, |\Omega| \qquad \forall\ \xi\in\rz^d.
\end{equation} 
Note also that, by the Pythagoras theorem, we have
\begin{equation} \label{r2}
R_2(\Lambda, \xi) = (2\pi)^{-d} \int_{\Omega} \big |e^{-i x\cdot\xi}-  (2\pi)^{d/2} \sum_{j \leq n(\Lambda)} \hat u_j(\xi)\, u_j(x)\big |^2\, dx.
\end{equation}
Our aim is to estimate $R_2(\Lambda. \xi)$ from below by a function of $\Lambda$, uniformly in $\xi$. 
Since $|a-b|^2 \geq \frac 12 |a|^2 -|b|^2$ for all $a,b\in \cz$, from \eqref{r2} it follows that for any $\beta>0$ 
\begin{align}
R_2(\Lambda,\xi) & \geq (2\pi)^{-d}\ \int_{\Omega_\beta} \big |e^{-i x\cdot\xi}-  (2\pi)^{d/2} \sum_{j \leq n(\Lambda)} \hat u_j(\xi)\, u_j(x)\big |^2\, dx \label{shell}  \\
& \geq \frac 12 \, (2\pi)^{-d}\, |\Omega_\beta| - \int_{\Omega_\beta} |F_\Lambda(\xi, x)|^2\, dx ,\nonumber
\end{align}
where we used the shorthand 
$$
F_\Lambda(\xi, x) =  \sum_{j \leq n(\Lambda)} \hat u_j(\xi)\, u_j(x).
$$
Since $F_\Lambda(\xi, \cdot)\in H^1_0(\Omega)$ for each $\Lambda>0$ and each $\xi\in\rz^d$, the Hardy inequality \eqref{hardy} in combination with \eqref{parseval} gives
\begin{align}
\int_{\Omega_\beta} |F_\Lambda(\xi, x)|^2\, dx & \leq \beta^2 \int_{\Omega_\beta} \frac{|F_\Lambda(\xi, x)|^2}{\delta^2(x)}\, dx \leq \beta^2 \int_{\Omega} \frac{|F_\Lambda(\xi, x)|^2}{\delta^2(x)}\, dx \nonumber\\
& \leq \beta^2\, c_h(\Omega)\, \int_{\Omega} |\nabla_x F_\Lambda(\xi,x)|^2\, dx = \beta^2\, c_h(\Omega)\, \sum_{j \leq n(\Lambda)} \lambda_j(\Omega)\ |\hat u_j(\xi)|^2 \nonumber \\
& \leq \beta^2\, \Lambda\,   c_h(\Omega)\, (2\pi)^{-d}\, |\Omega|. \label{upperb-1}
\end{align}
Hence in view of \eqref{shell} and \eqref{upperb-1} we get
\begin{equation} \label{lowerb-1}
R_2(\Lambda,\xi) \geq (2\pi)^{-d} \Big(\frac 12\, \frac{|\Omega_\beta|}{\beta} - \Lambda\, \beta\, c_h(\Omega)\, |\Omega|\Big)\, \beta.
\end{equation}
Now we choose 
\begin{equation} \label{beta}
\beta = \frac{\sigma(\Omega)}{4\, c_h(\Omega) \, |\Omega|}\ \Lambda^{-1},
\end{equation}
where $c_h(\Omega)$ is the constant from the Hardy inequality \eqref{hardy}. Note that the latter implies  
\begin{equation} \label{lambda1}
\lambda_1(\Omega) \geq \frac{1}{c_h(\Omega)\, R_i^2(\Omega)}.
\end{equation}
Using the definition of $\sigma(\Omega)$ we then find that for any $\Lambda \geq \lambda_1(\Omega)$ it holds
\begin{equation} \label{beta2}
\beta \leq \frac{\sigma(\Omega)}{4\, c_h(\Omega) \, |\Omega|}\ \lambda_1^{-1}(\Omega)\leq \frac{1}{4\, c_h(\Omega) \, R_i(\Omega)}\ \lambda_1^{-1}(\Omega) \leq \frac{R_i(\Omega)}{4}.
\end{equation}
From  \eqref{em} it thus follows that with our choice of $\beta$ we have 
$$
\frac{|\Omega_\beta|}{\beta} \, \geq \, \sigma(\Omega).
$$
Inserting the above estimate together with \eqref{beta} into \eqref{lowerb-1} we obtain
\begin{equation} \label{lowerb-2}
R_2(\Lambda,\xi)\  \geq \ \frac{1}{16\, c_h(\Omega)}\ (2\pi)^{-d}\, \frac{\sigma^2(\Omega)}{|\Omega|}\ \Lambda^{-1}.
\end{equation}

\medskip

\begin{proposition} \label{prop:ber}
For any $\Lambda \geq \lambda_1(\Omega)$ it holds
\begin{equation} \label{ber-imp}
\sum_{j: \lambda_j(\Omega)< \Lambda}\ (\Lambda-\lambda_j(\Omega)) \ \leq \ L^{cl}_{1,d}\, |\Omega|\ \Lambda^{1+\frac d2} -\frac{L^{cl}_{1,d}}{16\, c_h(\Omega)}\ \frac{\sigma^2(\Omega)}{|\Omega|}\ \Lambda^{\frac d2}\ ,
\end{equation}
where
\begin{equation} \label{lt-cl}
L^{cl}_{1,d} = \frac{1}{2^d\, \pi^{d/2}\, \Gamma(2+d/2)} .
\end{equation}
\end{proposition}

\begin{proof}
Since $R_1(\Lambda, \xi) \geq 0$, equations \eqref{1-eq} and \eqref{parseval} imply
$$
\sum_{j: \lambda_j(\Omega)< \Lambda}\ (\Lambda-\lambda_j(\Omega)) \ \leq \ (2\pi)^{-d}\ |\Omega|  \int_{\rz^d} (\Lambda-|\xi|^2)_+\, d\xi - 
\int_{\rz^d} (\Lambda-|\xi|^2)_+\, R_2(\Lambda,\xi)\, d\xi. 
$$
The claim now follows by inserting the lower bound \eqref{lowerb-2} and integrating with respect to $\xi$. 
\end{proof}

\noindent Note that the right hand side of \eqref{ber-imp} is positive for all $\Lambda \geq \lambda_1(\Omega)$ in view of inequality \eqref{lambda1} and $\sigma(\Omega) \leq |\Omega|/R_i(\Omega)$.

\begin{proof}[Proof of Theorem \ref{thm:li-yau}]
From \eqref{ber-imp} it follows that 
\begin{align*}
\sum_{j: \lambda_j(\Omega)< \Lambda}\ (\Lambda-\lambda_j(\Omega)) & \leq  L^{cl}_{1,d}\, |\Omega|\ \Lambda^{1+\frac d2}\Big(1- \frac{1}{16\, c_h(\Omega)}\ \frac{\sigma^2(\Omega)}{|\Omega|^2\ \Lambda}\Big) \\
&  \leq  L^{cl}_{1,d}\, |\Omega|\ \Lambda^{1+\frac d2}\left(1- \frac{1}{16\, c_h(\Omega)}\ \frac{\sigma^2(\Omega)}{|\Omega|^2\ \Lambda}\right)^{1+\frac d2} \\
& =  L^{cl}_{1,d}\, |\Omega|\ \left(\Lambda- \frac{1}{16\, c_h(\Omega)}\ \frac{\sigma^2(\Omega)}{|\Omega|^2}\right)^{1+\frac d2}. 
\end{align*}
Since both sides of the above inequality are convex functions of $\Lambda$, we can apply the Legendre transform. This yields \eqref{li-yau-imp}.
\end{proof}

\smallskip
\subsubsection*{Convex domains}

\begin{lemma} \label{lem:convex}
Let $\Omega\subset\rz^d$ be  bounded and convex. Then 
\begin{equation} \label{m-convex}
\sigma(\Omega) = \frac{|\Omega|}{R_i(\Omega)} \, .
\end{equation}
\end{lemma}

\begin{proof}
Let us first prove the statement for domains with $C^1$ boundary.
We are going to show that 
$$
f(\beta) = \frac{|\Omega_\beta|}{\beta}
$$ 
is a decreasing function of $\beta$ on $(0, R_i(\Omega))$. To this end let $\beta_0\in(0, R_i(\Omega))$ and consider the sets 
$$
E_0 = \{ x\in\Omega\, : \, \delta(x) \geq \beta_0\}, \quad \text{and} \quad
E_t = \{ x\in\Omega\setminus E_0\, : \, {\rm dist} (x, E_0) \leq t\}, \quad t>0.
$$
From the convexity of $\Omega$ it follows that $\delta$ is concave and therefore $E_0$ is a compact convex set. Hence by the Steiner formula, see e.g. \cite{HHL}, it holds 
\begin{equation} \label{steiner}
|E_t| = \sum_{j=0}^d\, K_j(E_0)\ t^j,
\end{equation}
where $K_j(E_0)$ are non-negative coefficients depending on the geometry of $E_0$. We claim that 
\begin{equation} \label{union}
E_{\beta_0-\beta} \cup E_0  = \Omega^c_\beta, \qquad  0 < \beta<\beta_0,
\end{equation}
where $\Omega^c_\beta= \Omega\setminus \Omega_\beta$ is the complement of $\Omega_\beta$ in $\Omega$. 
Indeed,  let $y\in\partial E_0$ and denote by $r_y$ the half-line emanating from $y$ perpendicularly to the tangent plane of $\partial E_0$ at $y$. Let $z_y\in\partial\Omega$ be given by the intersection of $\partial\Omega$ and $r_y$. Since $\delta(y)=\beta_0$ we have 
\begin{equation} \label{zy}
{\rm dist} (y, z_y) = \delta(y) =\beta_0, \qquad  y\in\partial E_0.
\end{equation} 
Now let $x\in\Omega^c_\beta$. Then there exists an $y(x)\in\partial E_0$ such that $x\in r_{y(x)}$. Hence
$$
{\rm dist} (y(x), x) = \delta(y(x)) - {\rm dist} (x, z_{y(x)}) = \beta_0-  {\rm dist} (x, z_{y(x)}) \leq \beta_0-\delta(x) \leq\beta_0-\beta.
$$
This implies that $\Omega^c_\beta \subseteq  E_{\beta_0-\beta} \cup E_0$. To prove the opposite inclusion, let $x\in (E_{\beta_0-\beta} \cup E_0)$. By the triangle inequality and \eqref{zy} 
$$
\beta_0 \leq {\rm dist} (x, E_0) + \delta(x) \leq \beta_0-\beta +\delta(x),
$$
which shows that $x\in \Omega_\beta^c$. Therefore \eqref{union} holds true and consequently 
\begin{equation} \label{cc}
 |\Omega_\beta| = |\Omega| -  |E_{\beta_0-\beta} \cup E_0| .
\end{equation}
In view of \eqref{steiner} it follows that $|E_{\beta_0-\beta} \cup E_0|$ is a convex function of $\beta$. Hence $|\Omega_\beta| $ is a concave function of $\beta$ on $(0,\beta_0)$, see \eqref{cc}, and since $|\Omega_0|=0$, we easily verify that $f(\beta)=|\Omega_\beta|/\beta$ is decreasing on $(0,\beta_0)$ for any $\beta_0 < R_i(\Omega)$ . This proves the statement of the Lemma for $C^1$ smooth domains.

\smallskip

If $\partial\Omega$ is not $C^1$, then we approximate $\Omega$ by a sequence of domains $\Omega^n$ with $C^1$ smooth boundary and such that the Hausdorf distance between $\Omega$ and $\Omega^n$ tends to zero as $n\to \infty$. Then
$$
f(\beta) = \lim_{n\to\infty} \frac{|\Omega^n_\beta|}{\beta}. 
$$ 
Since a pointwise limit of a sequence of decreasing functions is a decreasing function, we again conclude that $f(\beta)$ is decreasing. This completes the proof.
\end{proof}

\begin{proof}[Proof of Corollary \ref{li-yau-conv}]
The claim follows from Theorem \ref{thm:li-yau}, Lemma \ref{lem:convex} and the fact that for convex domains  $c_h(\Omega) =4$ independently of $\Omega$, \cite{Da, MS}.
\end{proof}


\subsection{Magnetic Dirichlet Laplacian}
Let $P_k$ be the orthogonal projection onto the $k$th Landau level $B(2k-1)$
of the Landau Hamiltonian with constant magnetic field $B$ in $L^2(\rz^2)$. Denote by $P_k(x,y)$ the integral kernel of $P_k$.  Note that
\begin{align}
P_k(x,x)&=\frac{1}{2\pi} B\,, \label{diag} \\
\int_{\rz^2} \Big( \int_{\Omega}  |P_k(y,x)|^2\, dx \Big)\, dy &=\int_{\Omega} \Big( \int_{\rz^2}  P_k(y,x)\overline{P_k(x,y)}\, dy \Big)\, dx \label{P2} \\
&= \int_{\Omega} P_k(x,x)\, dx =\frac{B}{2\pi}|\Omega|\,.\notag
\end{align}
Let $\phi_j$ be the normalised eigenfunctions of $H_\Omega(A)$ corresponding to the eigenvalues $\lambda_j(\Omega;A)$. 
Put 
$$
f_{k,j}(y)=\int_{\Omega}P_k(y,x)\phi_j(x) dx, \qquad y\in\rz^2.
$$ 
Our goal is to establish an analog of Proposition \ref{prop:ber} for magnetic Dirichlet Laplacians on planar domains. Let $\Lambda>0$. We have
\begin{align*}
\sum_{j: \lambda_j(\Omega;A)<\Lambda}(\Lambda-\lambda_j(\Omega;A))
&=
\sum_{j: \lambda_j(\Omega;A)<\Lambda} 
(\Lambda\|\phi_j\|^2_{L^2(\Omega)}-\|(i\nabla_{x}+A)\phi_j\|^2_{L^2(\Omega)})
\\
&=\sum_{j: \lambda_j(\Omega;A)<\Lambda} \sum_{k\in\nz}
(\Lambda\|f_{k,j}\|^2_{L^2(\rz^2)}-\|(i\nabla_{x}+A)f_{k,j}\|^2_{L^2(\rz^2)})\\
&=\sum_{j: \lambda_j(\Omega;A)<\Lambda} \sum_{k\in\nz}
(\Lambda-B(2k-1))\|f_{k,j}\|^2_{L^2(\rz^2)}\,.
\end{align*}
In analogy with the procedure in the non-magnetic case we split 
\begin{align} \label{sum-1}
\sum_{j: \lambda_j(\Omega;A)<\Lambda}(\Lambda-\lambda_j(\Omega;A))
&=\sum_{j: \lambda_j(\Omega;A)<\Lambda} \  \sum_{k\in\nz}
(\Lambda-B(2k-1))\|f_{k,j}\|^2_{L^2(\rz^2)}\\
&=\sum_{j: \lambda_j(\Omega;A)<\Lambda} \ \sum_{k: \Lambda>B(2k-1)}
(\Lambda-B(2k-1))\|f_{k,j}\|^2_{L^2(\rz^2)}\nonumber\\
&+\sum_{j: \lambda_j(\Omega;A)<\Lambda} \ \sum_{k: \Lambda\leq B(2k-1)}
(\Lambda-B(2k-1))\|f_{k,j}\|^2_{L^2(\rz^2)}\nonumber\\
&=\sum_{k: \Lambda>B(2k-1)}
(\Lambda-B(2k-1))\sum_{j\in\nz}\|f_{k,j}\|^2_{L^2(\rz^2)}\nonumber\\
&- \sum_{k: \Lambda\leq B(2k-1)} (B(2k-1)-\Lambda)\, R_1(\Lambda,k)\nonumber\\
&-\sum_{k: \Lambda>B(2k-1)} (\Lambda-B(2k-1))\, R_2(\Lambda,k) \,,\nonumber
\end{align}
where 
\begin{align*}
R_1(\Lambda,k)&=\sum_{j: \lambda_j(\Omega;A)<\Lambda} 
\|f_{k,j}\|^2_{L^2(\rz^2)}\, ,  \qquad 
R_2(\lambda,k) =\sum_{j: \lambda_j(\Omega;A)\geq \Lambda} 
\|f_{k,j}\|^2_{L^2(\rz^2)}\,.
\end{align*}
By Parseval's identity and equation \eqref{P2} it follows that for all $\Lambda>0$ and all $k\in\nz$ we have 
\begin{align}
 \sum_{j\in\nz}\|f_{k,j}\|^2_{L^2(\rz^2)}&= R_1(\Lambda,k)+R_2(\Lambda,k)
 = \int_{\rz^2} \Big|\sum_{j\in\nz}  \int_\Omega P_k(y,x)\phi_j(x)dx\, \Big|^2dy \label{parseval-mg}  \\
& = \int_{\rz^2}  \int_\Omega |P_k(y,x)|^2 \, dx dy=\frac{B}{2\pi}|\Omega|\,. \nonumber
\end{align}
Let 
\begin{equation} \label{qk}
Q_k(x,y;\Lambda) = \sum_{j: \lambda_j(\Omega;A)<\Lambda} f_{k,j}(y)\, \overline{\phi}_j(x).
\end{equation}
We now use identities \eqref{diag}-\eqref{P2} to find that, similarly as in section \ref{sec:non-mag},  for any $\beta\leq R_i(\Omega)$ it holds
\begin{align} \label{pyth-mg} 
R_2(\Lambda, k) & =  \int_{\rz^2} \left(  \int_{\Omega} \big |P_k(x,y)- Q_k(x,y;\Lambda) \big |^2\, dx \right)\, dy  \\
& \geq \frac 12 \int_{\rz^2} \int_{\Omega_\beta} |P_k(x,y)|^2 dx \, dy  -  \int_{\rz^2} \int_{\Omega_\beta} |Q_k(x,y;\Lambda)|^2\, dx \, dy \nonumber \\
&= \frac{B}{4\pi}\, |\Omega_\beta| - \int_{\rz^2} \int_{\Omega_\beta} |Q_k(x,y;\Lambda)|^2\, dx \, dy \nonumber.
\end{align}
Since $Q_k(\cdot,y;\Lambda)\in H^1_0(\Omega)$ for all $k\in\nz, y\in\rz^2$ and $\Lambda>0$, the Hardy inequality \eqref{hardy} in combination with \eqref{diamag} yield 
\begin{align*}
\int_{\Omega_\beta} |Q_k(x,y;\Lambda)|^2\, dx & \leq \beta^2 \, \int_{\Omega_\beta} \frac{|Q_k(x,y;\Lambda)|^2}{\delta^2(x)}\, dx  \leq 
\beta^2 \, \int_{\Omega} \frac{|Q_k(x,y;\Lambda)|^2}{\delta^2(x)}\, dx \\
& \leq  \beta^2\, c_h(\Omega)  \int_\Omega | (i\nabla_x+A)\, Q_k(x,y;\Lambda) |^2 \, dx\\
& = \beta^2\, c_h(\Omega)  \sum_{j: \lambda_j(\Omega;A)<\Lambda} |f_{k,j}(y)|^2\, \lambda_j(\Omega;A)\  \leq \ \beta^2\, c_h(\Omega)\, \Lambda \sum_{j: \lambda_j(\Omega;A)<\Lambda} |f_{k,j}(y)|^2. 
\end{align*}
By inserting the above estimate into \eqref{pyth-mg} and using \eqref{parseval-mg} again we obtain 
\begin{align*} 
R_2(\Lambda, k) & \geq  \frac{B}{4\pi}\, |\Omega_\beta| - \beta^2\, c_h(\Omega)\, \Lambda \sum_{j: \lambda_j(\Omega;A)<\Lambda} \|f_{k,j}\|_{L^2(\rz^2)}^2\\
& \geq \frac{B}{4\pi}\, \Big( \frac{|\Omega_\beta|}{\beta} - 2 \beta\, \Lambda \, c_h(\Omega)\, |\Omega| \Big)\, \beta. 
\end{align*}
Note that in view of \eqref{diamag} we have
\begin{equation} \label{mu1}
\lambda_1(\Omega;A) \geq \lambda_1(\Omega).
\end{equation}
Hence choosing  $\beta$ as in \eqref{beta} and following the reasoning in \eqref{beta2} we conclude that $\beta \leq R_i(\Omega)/4$ and therefore $\frac{|\Omega_\beta|}{\beta} \geq \sigma(\Omega)$. This implies
\begin{equation} \label{lb-mg1} 
R_2(\Lambda, k) \ \geq  \  \frac{B}{32\, \pi\, c_h(\Omega)}\ \frac{\sigma^2(\Omega)}{ |\Omega|}\ \Lambda^{-1} \qquad\forall\ k\in\nz.	 
\end{equation}

\begin{proposition} \label{prop:ber-mg}
Let $d=2$. For any $\Lambda \geq \lambda_1(\Omega;A)$ it holds
\begin{equation} \label{eq-ber-mg}
\sum_{j: \lambda_j(\Omega;A)< \Lambda}\ (\Lambda-\lambda_j(\Omega;A)) \ \leq \ \frac{ |\Omega|}{8\, \pi} \ \Lambda^2  -\frac{1}{128\, \pi\, c_h(\Omega)} \ \frac{\sigma^2(\Omega)}{|\Omega|}\ \Lambda. 
\end{equation}
\end{proposition}

\begin{proof}
Put $M=\left[\frac{\Lambda}{2B}+\frac12\right]$ and 
$m=\left\{\frac{\Lambda}{2B}+\frac12\right\}$ and thus
$M+m= \frac{\Lambda}{2B}+\frac12$. Then
\begin{align*}
\sum_{k: \Lambda>B(2k-1)}(\Lambda-B(2k-1))
&=
M\Lambda-BM^2=B M\left(\frac{\Lambda}{B}-M\right)\\
&=B\left( \frac{\Lambda}{2B}+\frac12-m\right)\left(\frac{\Lambda}{2B}-\frac12+m\right)\\
&=B\left(\frac{\Lambda^2}{4B^2}-\left(\frac12-m\right)^2\right)\,.
\end{align*}
Since $R_1(\Lambda,k) \geq 0$, the above identity together with \eqref{sum-1} and \eqref{parseval-mg} implies
\begin{align*}
\sum_{j: \lambda_j(\Omega;A)< \Lambda}\ (\Lambda-\lambda_j(\Omega;A)) &\  \leq \  \frac{ |\Omega|}{8\, \pi} \ \Lambda^2 -\frac{1}{128\,  c_h(\Omega) \, \pi} \ \frac{\sigma^2(\Omega)}{|\Omega|}\ \Lambda \\
& \quad - B^2 \Big(\frac 12-m\Big)^2\, \Big(\frac{|\Omega|}{2 \pi}-\frac{1}{32\, \pi\, c_h(\Omega)}\ \frac{\sigma^2(\Omega)}{ |\Omega|\, \Lambda}\Big).
\end{align*}
The last term on the right hand side of the last inequality is negative since $\Lambda\, c_h(\Omega) \geq \lambda_1(\Omega)\, c_h(\Omega) \geq R_i^{-2}(\Omega)$, by \eqref{hardy} and \eqref{mu1}, and $\sigma(\Omega) \leq |\Omega|/R_i(\Omega)$. The claim now follows. 
\end{proof}

\begin{proof}[Proof of Theorem \ref{thm:li-yau-mg}] 
Inequality \eqref{eq-li-yau-mg} now follows from Proposition \ref{prop:ber-mg} by the Legendre transformation in the same as in the case of the Dirichlet Laplacian.
\end{proof}

\noindent Corollary \ref{cor:convex} is a consequence of Theorem \ref{thm:li-yau-mg} and Lemma \ref{lem:convex}.

\section{\bf Further improvements} 
\label{sec-imp}

\noindent The order of the reminder term in \eqref{eq-ber-mg} can further be improved applying a straightforward generalization of a result by Davies, \cite{Da2}. We are grateful to Rupert Frank who pointed this fact out to us.

\begin{proposition} \label{prop-davies}
Let $\Omega\subset\rz^d$ be an open bounded set. Let $A\in C(\overline{\Omega}, \rz^2)$ and let $H(A)$ be the associated magnetic Dirichlet Lalpacian in $L^2(\Omega)$.  Assume that  the Hardy inequality
\begin{equation}  \label{hardy-mag}
\int_\Omega |i\nabla u + A u |^2 \,dx \,  \geq  \, c^{-2} \int_\Omega \frac{|u|^2}{\delta^2} \,dx \,,
\qquad u\in C_0^\infty(\Omega) \,,
\end{equation}
holds for some $c\geq 2$. Then for every $\beta>0$,
\begin{equation} \label{bdry}
\int_{\Omega_\beta} |u|^2 \,dx \, \leq  \, c^{2+2/c}\,  \beta^{2+2/c}\  \| H(A) \, u\| \, \| H(A)^{1/c} u\|
\end{equation}
for any $u$ in the operator domain of $H(A)$.
\end{proposition}

\smallskip

\noindent Proposition \ref{prop-davies} was proved in \cite{Da2} for the case $A=0$. However, a detailed inspection of the proof 
of \cite[Thm.~4]{Da2} shows that the same method applies also to the magnetic Dirichlet Laplacian. We then have

\begin{theorem} \label{thm-rupert}
Let $\Omega\subset\rz^2$ be open and bounded. Let $A=\frac 12 (- B x_2, B x_1)$. Then for any $\Lambda\geq \lambda_1(\Omega;A)$ it holds
\begin{equation} \label{eq-rupert}
\sum_{j: \lambda_j(\Omega;A)< \Lambda}\ (\Lambda-\lambda_j(\Omega;A)) \ \leq \ \frac{ |\Omega|}{8\, \pi} \ \Lambda^2  -
K(\Omega)\, \sigma(\Omega) \left(\frac{\sigma(\Omega)}{|\Omega|}\right)^{\frac{2}{2+\mu}}\ \Lambda^{\frac{3+\mu}{2+\mu}}, 
\end{equation}
where $\mu = \mu(\Omega) = \sqrt{c_h(\Omega)}\, $ and 
$$
K(\Omega)= \frac{2+\mu}{16\, \pi\, \mu}\, (2+2\mu)^{-\frac{2+3\mu}{2+\mu}}.
$$
\end{theorem}

\begin{proof}
Let us fix $k\in\nz$ and $y\in\rz^2$. 
Since $Q_k(\cdot,y;\Lambda)$ belongs to the domain of $H(A)$ for any $\Lambda>0$, see equation \eqref{qk}, 
we can apply inequality \eqref{bdry}, with $c= \mu=\sqrt{c_h(\Omega)}$, to the function $u= Q_k(\cdot,y;\Lambda)$. This yields 
$$
\int_{\Omega_\beta} |Q_k(x,y;\Lambda)|^2\, dx \, \leq \, (\mu\, \beta)^{2+\frac 2\mu} \ \Lambda^{1+\frac 1\mu}  \sum_{j: \lambda_j(\Omega;A)<\Lambda} |f_{k,j}(y)|^2. 
$$
If we now insert the above bound into \eqref{pyth-mg} and keep in mind that 
$$
\sum_{j: \lambda_j(\Omega;A)<\Lambda} \|f_{k,j}\|_{L^2(\rz^2)}^2 \, \leq \, \frac{B}{2\pi}\, |\Omega|
$$
by \eqref{parseval-mg}, we find that  
$$
R_2(\Lambda, k)  \geq \frac{B}{4\pi}\, \Big( \frac{|\Omega_\beta|}{\beta} - 2 \, \mu^{2+\frac 2\mu}\,  \beta^{1+\frac 2\mu} \ \Lambda^{1+\frac 1\mu} \,  |\Omega| \Big)\, \beta. 
$$
Optimizing the right hand side with respect to $\beta$ gives
\begin{equation} \label{lb-mg2}
R_2(\Lambda, k)  \, \geq \, B\, K(\Omega) \, \sigma(\Omega)\, \Big(\frac{\sigma(\Omega)}{|\Omega|}\Big)^{\frac{\mu}{2+\mu}}\,  \Lambda^{-\frac{1+\mu}{2+\mu}} .
\end{equation}
We now follow the arguments of the proof of Proposition \ref{prop:ber-mg} with the lower bound \eqref{lb-mg1} replaced by \eqref{lb-mg2} and arrive at \eqref{eq-rupert}. 
\end{proof}

\begin{remark} 
The power of $\Lambda$ in the reminder term of \eqref{eq-rupert} is larger than the one of \eqref{eq-ber-mg} by factor $\frac{1}{2+\mu}$. 
\end{remark}

\noindent For convex domains inequality \eqref{hardy-mag} holds true with $c=2$. Hence Theorem \ref{thm-rupert} in combination with Lemma \ref{lem:convex} implies

\begin{corollary}
Let $\Omega\subset\rz^2$ be bounded and convex and let $A=\frac 12 (- B x_2, B x_1)$. Then
\begin{equation*}
\sum_{j: \lambda_j(\Omega;A)< \Lambda}\ (\Lambda-\lambda_j(\Omega;A)) \ \leq \ \frac{ |\Omega|}{8\, \pi}\,  \left( \Lambda^2  -
\frac{\Lambda^{5/4}  }{36\, R_i(\Omega)^{3/2}} \right) .  
\end{equation*}
\end{corollary}

\medskip

\section*{\bf Acknowledgements}
\noindent 
The work on this paper was initiated by the programme
ÒHamiltonians in Magnetic FieldsÓ in fall 2012 at the Institut Mittag- Leffler, Djursholm, Sweden. 
H.K. has been partially supported by the MIUR-PRIN 2010-11 grant for the project  ''Calcolo delle Variazioni''. T.W. has been partially supported by the DFG  project WE 1964/4-1.



\begin{thebibliography}{99}

%
\bibitem{A} A.~Ancona : On strong barriers and inequality of Hardy for domains in $\rz^n$. 
{\em J. London Math. Soc.} \textbf{34} (1986), 274--290.
%
\bibitem{Be} F.A.~Berezin: Covariant and contravariant symbols of
operators. {\em Izv.~Akad.~Nauk SSSR Ser.~Mat.} {\bf 36} (1972)
 1134--1167. 
%
\bibitem{Da} E.B.~Davies:  A review of Hardy inequalities. The Maz'ya anniversary collection, Vol. 2, {\em Oper. Theory Adv. Appl.}  Birkh\"auser, Basel, {\bf 110} (1999) 55--67. 
%
\bibitem{Da2} E. B.~Davies: Sharp boundary estimates for elliptic operators. {\em Math. Proc. Camb. Phil. Soc.}\textbf{129} (2000) 165-178.
%
\bibitem{ELV}
L.~Erd\"os, M.~Loss  and V.~Vougalter:
Diamagnetic behaviour of sums of Dirichlet eigenvalues. {\em Ann. Inst. Fourier}  {\bf 50}  (2000) 891-907.
%
\bibitem{FG} R.~Frank, L.~Geisinger: Two-term spectral asymptotics for the Dirichlet Laplacian on a bounded domain. In: {\em Mathematical results in quantum physics, Proceedings of the QMath11 Conference}, P. Exner (ed.),  World Scientific, Singapore, 2011, 138--147. 
%
\bibitem{FLW}
R.~Frank, M.~Loss and T.~Weidl: Polya's conjecture in the presence of a constant magnetic field.
{\em J. Eur. Math. Soc. (JEMS)}  {\bf 11} (2009)  1365-1383.
%
\bibitem{FLU} J.K.~Freericks, E.H.~Lieb, D.~Ueltschi: Segregation in the Falicov-Kimball model. {\em  Comm.~Math.~Phys.} {\bf 227} (2002)
243--279.
%
\bibitem{GLW} L.~Geisinger, A.~Laptev and T.~Weidl: Geometrical versions of improved Berezin-Li-Yau inequalities.
{\em J. Spectr. Theory} {\bf 1}  (2011) 87--109.
%
\bibitem{HHL}  M.~Heveling, D.~Hug, G.~Last: Does polynomial parallel volume imply convexity? {\em Math. Ann.} {\bf 328} (2004) 469--479.
%
\bibitem{Iv} V.~Ivrii: {\em Microlocal analysis and precise spectral asymptotics.}. 
Springer Monographs in Mathematics. Springer-Verlag, Berlin, 1998.
%
\bibitem{Iv2} V.~Ivrii: The second term of the spectral asymptotics for a Laplace-Beltrami operator on manifolds with boundary. (Russian) {\em Funktsional. Anal. i Prilozhen}  {\bf 14}  (1980), no. 2, 25--34.  
%
\bibitem{KVW} H.~Kova\v r\'{i}k, S.~Vugalter, T.~Weidl : Two dimensional Berezin-Li-Yau inequalities with a correction term. {\em Commun. Math. Phys.}  {\bf 287} (2009) 959-981.
%
\bibitem{Lap} A.~Laptev : Dirichlet and Neumann Eigenvalue Problems 
on Domains in Euclidean Spaces. {\em J. Func. Anal.}  {\bf 151} (1997)  531--545.
%
\bibitem{LW} A.~Laptev, T.~Weidl :  Sharp Lieb--Thirring inequalities in high dimensions. {\em Acta Math.} {\bf 184} (2000)  87--111.
%
\bibitem{LWHar}
A.~Laptev  and T.~Weidl: Hardy inequalities for magnetic Dirichlet forms. {\em Oper. Theory Adv. Appl.} {\bf 108}
(1999) 299--305.
%
\bibitem{LY} P.~Li  and S.T.~Yau: On the Schr\"odinger equation and the eigenvalue problem. 
{\em Comm. Math. Phys.}  {\bf 88}, (1983)  309--318.
%
\bibitem{LL} E. H.~Lieb, M.~Loss, \textit{Analysis}.
Second edition. Graduate Studies in Mathematics \textbf{14},
American Mathematical Society, Providence, RI, 2001.
%
\bibitem{MS} T.~Matskewich and P. E.~Sobolevskii: The best possible constant in a generalized Hardy's
inequality for convex domains in $\rz^n$, {\em Nonlinear Analysis TMA}, {\bf 28} (1997), 1601--1610.
%
\bibitem{Me} A.D.~Melas: A lower bound for sums of eigenvalues of the
Laplacian. {\em Proc. Amer. Math. Soc.}  {\bf 131} (2003) 631--636.
%
\bibitem{Po} G.~P\'olya: On the eigenvalues of vibrating membranes.
{\em  Proc.~London Math.~Soc.} {\bf 11}  (1961) 419--433. 
%
\bibitem{SV} Yu.~Safarov, D.~Vassiliev:
{\em The asymptotic distribution of eigenvalues of partial differential
operators.} Translations of Mathematical Monographs, 155. American
Mathematical Society, Providence, RI, 1997.
%
\bibitem{W}
T.~Weidl: Improved Berezin--Li--Yau inequalities with a remainder term. {\em Amer. Math.
Soc. Transl.}  {\bf 225} (2008) 253-263.
%
\bibitem{We} H.~Weyl: Das asymptotische Verteilungsgesetz der
  Eigenwerte linearer partieller Differentialgleichungen. {\em Math.~Ann.} {\bf 71} (1912) 441--479.
%
\bibitem{Y} S. Y.~Yolcu: An improvement to a Berezin-Li-Yau type inequatliy. {\em Proc. Amer. Math. Soc.} {\bf 138} (2010), 4059--4066.
%
\bibitem{YY} S. Y.~Yolcu and T.~Yolcu: A Berezin-Li-Yau type inequality for the fractional Laplacian on a bounded domain. {\em Rev. Math. Phys.} {\bf 24}  (2012) 18 pp.
%
\end{thebibliography}
\end{document}